\documentclass[12pt]{article}
\usepackage[a4paper, margin=2.3cm]{geometry}
\usepackage{amssymb,amsthm, amsmath}
\usepackage{hyperref}
\usepackage{color}

\newtheorem{theorem}{Theorem}[section]
\newtheorem{corollary}[theorem]{Corollary}

\newtheorem{conjecture}[theorem]{Conjecture}
\newtheorem{construction}[theorem]{Construction}
\theoremstyle{definition}

\newcommand{\R}{\mathbb R}
\newcommand{\C}{\mathbb C}

\author{Frank de Zeeuw}
\title{Spanned lines and Langer's inequality}

\begin{document}
\date{}
\maketitle

\begin{abstract}
We collect some results in combinatorial geometry that follow from an inequality of Langer in algebraic geometry.
Langer's inequality gives a lower bound on the number of incidences between a point set and its spanned lines, and was recently used by Han to improve the constant in the weak Dirac conjecture.
Here we observe that this inequality also leads to improved constants in Beck's theorem,
which states that a finite point set in $\R^2$ or $\C^2$ has many points on a line or spans many lines. 

Most of the proofs that we use are not original, and the goal of this note is mainly to carefully record the quantitative results in one place.
We also include some discussion of possible further improvements to these statements.
\end{abstract}

\section{Introduction}

We consider the following standard notions from combinatorial geometry.
For a point set $P$ in a space $\R^d$ or $\C^d$, we write $L(P)$ for the set of lines spanned by $P$,
and we set $n = |L(P)|$.
We write $\ell_i$ for the number of lines in $L(P)$ containing exactly $i$ points of $P$.
We have the basic equalities
\begin{equation}\label{eq:basic}
\sum_{i\geq 2} \ell_i = n,~~~ 
\sum_{i\geq 2}i\ell_i 
= I(P,L(P)),
~~~\text{and}~~~
\sum_{i\geq 2} \binom{i}{2}\ell_i = \binom{n}{2},
\end{equation}
where 
$I(P,L(P))$ denotes the number of incidences between $P$ and $L(P)$, 
i.e., the number of pairs $(p,\ell)\in P\times L(P)$ such that $p\in \ell$.

We consider the implications of the following result of Langer \cite[Proposition 11.3.1]{L}. 

\begin{theorem}[Langer]\label{thm:langer}
Let $P$ be a set of $n$ points in $\C^2$,
with at most $2n/3$ points collinear.
Then
\begin{equation}\label{eq:langer}
\sum_{i\geq 2} i \ell_i\geq \frac{n(n+3)}{3}.
\end{equation}
\end{theorem}

Langer derived this theorem from his generalization of the Bogomolov--Miyaoka--Yau inequality for algebraic surfaces (we will not try to explain the underlying theory here, and we will consider Theorem \ref{thm:langer} as a black box).
Langer's work was published in 2003, but it appears that the inequality has long gone unnoticed in combinatorial geometry.
For instance, it is sharper than a later result of Payne and Wood \cite[Theorem 4]{PW}.
Recently, Han \cite{Ha} uncovered Theorem \ref{thm:langer} and used it to improve the constant in the weak Dirac conjecture.
To be precise, Han obtained \eqref{eq:langer} in the middle of the proof of \cite[Theorem 10]{Ha} as a consequence of a Hirzebruch-type inequality (stated as \eqref{eq:bojanowski1} below),
which can be found in the Master's thesis of Bojanowski \cite{Bo} 
(see also Pokora \cite{Po}).
This inequality from \cite{Bo} is in turn based on the work of  Langer \cite{L}, 
so Han's result is indirectly derived from \cite{L}, 
but it can also be obtained directly from Langer \cite[Proposition 11.3.1]{L}, i.e., from Theorem \ref{thm:langer}.

Let us introduce Han's improvement of the weak Dirac conjecture.
The strong Dirac conjecture \cite{D} states that there is a constant $c$ such that any set of $n$ non-collinear points in $\R^2$ contains a point that lies on at least $n/2-c$ lines spanned by the point set.
The weak Dirac conjecture states that there is a constant $c'>0$ such that any set of $n$ non-collinear points in $\R^2$ has a point on at least $c'n$ spanned lines.
This weak form was proved by Beck \cite{Be} and Szemer\'edi and Trotter \cite{ST}, in both cases with $c'$ unspecified but very small.
Payne and Wood \cite{PW} obtained $c' = 1/37$, 
Pham and Phi \cite{PP} improved it to $c'=1/26$,
and Han \cite{Ha} made a large jump to $c' = 1/3$. 
We include the proof from Theorem \ref{thm:langer} here for the sake of exposition.

\begin{corollary}[Han]\label{cor:han}
Let $P$ be a set of $n$ points in $\C^2$, not contained in a line.
Then there is a point in $P$ that is contained in at least $(n+3)/3$ lines spanned by $P$.
\end{corollary}
\begin{proof}
If there is a line $\ell\in L(P)$ containing more than $2n/3$ points of $P$, 
then any point off $\ell$ lies on more than $2n/3$ lines  of $L(P)$. Thus we can assume that $P$ has at most $2n/3$ points collinear, so that Theorem \ref{thm:langer} gives
\[I(P,L(P)) = \sum_{i\geq 2} i \ell_i\geq \frac{n(n+3)}{3}.\]
It follows that there is a point $p\in P$ that is involved in at least $(n+3)/3$ incidences, i.e., it lies on that many spanned lines.
\end{proof}

Remarkably, Theorem \ref{thm:langer} and Corollary \ref{cor:han} are both tight in $\C^2$, by Construction \ref{con:hesse} below.
On the other hand, that construction is not possible in $\R^2$, so there may be room for improvement over $\R$.

\paragraph{Hirzebruch-type inequalities.}
Let us clarify what is meant by a ``Hirzebruch-type inequality''.
For a point set $P$ in $\R^2$ that is not collinear, \emph{Melchior's inequality} \cite{M} states
\begin{equation}\label{eq:melchior}
\ell_2 \geq 3 + \sum_{i\geq 4} (i-3)\ell_i.
\end{equation}
This inequality, which can be derived from Euler's formula, has many applications in combinatorial geometry.
For instance, 
it implies the Sylvester--Gallai theorem, which says that $\ell_2>0$ for any non-collinear point set (see for instance \cite{GT}).
Melchior's inequality is false in $\C^2$, 
again because of Construction \ref{con:hesse} below.

For a point set $P$ in $\C^2$ (or in $\R^2$) with at most $n-3$ points collinear, \emph{Hirzebruch's inequality} \cite{Hi} states that
\begin{equation}\label{eq:hirzebruch}
\ell_2 + \frac{3}{4}\ell_3 \geq n + \sum_{i\geq 5} (2i-9)\ell_i.
\end{equation}
One particular consequence is that $\ell_2 +\ell_3>0$ for any $n$ points in $\C^2$ with at most $n-3$ collinear.
Kelly \cite{K} used this fact to solve a problem of Serre.

Bojanowski \cite{Bo} and Pokora \cite{Po} used Langer's work \cite{L} to prove an inequality of a similar form.
For a point set $P$ in $\C^2$ with at most $2n/3$ points collinear, we have 
\begin{equation}\label{eq:bojanowski1}
\ell_2 + \frac{3}{4}\ell_3 \geq n + \sum_{i\geq 5} \frac{i^2-4i}{4}\ell_i,
\end{equation}
or equivalently
\begin{equation}\label{eq:bojanowski2}
\sum_{i\geq 2} (4i-i^2)\ell_i \geq 4n.
\end{equation}
This inequality is equivalent to \eqref{eq:langer}:
Moving $\sum(i-i^2)\ell_i$ to the right in \eqref{eq:bojanowski2}, we get
\[ 3\sum_{i\geq 2} i\ell_i \geq 4n + \sum_{i\geq 2} (i^2-i)\ell_i 
=4n + 2\sum_{i\geq 2} \binom{i}{2} \ell_i
= 4n + 2\binom{n}{2}
=n^2+3n.
\]

The inequality \eqref{eq:bojanowski1} is at least as good as \eqref{eq:hirzebruch}, 
and strictly better whenever there is a line with at least five points.
Of course, its condition on the number of collinear points is also stronger,
but in combinatorial problems the case where many points are collinear can usually be handled in other ways.


\paragraph{Constructions.}
As remarked by Langer \cite[Example 11.3.2]{L}, 
the inequality \eqref{eq:langer} is tight in $\C^2$ for infinitely many $n$, because of the following configuration.

\begin{construction}[Fermat configuration]\label{con:hesse}
For $n$ divisible by $3$, there is a set $P$ of $n$ points in $\C^2$ with the following properties.
The points are contained in three non-concurrent lines, with $n/3$ points on each line.
For one point of $P$ on one of the three lines and another point of $P$ on another of the lines, the line through the two points hits the third line in a point of $P$.
Thus $P$ spans $(n/3)^2$ lines with three points, 
and none with two points.
Therefore we have
\[\sum_{i\geq 2}i\ell_i = \left(\frac{n}{3}\right)^2\cdot 3 + 3\cdot \frac{n}{3} = \frac{n(n+3)}{3},\]
showing that \eqref{eq:langer} is best possible.
Each point lies on exactly $(n+3)/3$ spanned lines, showing that Corollary \ref{cor:han} is tight.
In total there are $n^2/9+3$ spanned lines.
\end{construction}

For a more explicit description of the construction, see \cite[Example 1]{BDSW}.
Note that by adding or removing a few points, 
one can obtain constructions for any $n$ with a similar number of incidences (although one loses the property that no line has two points).
Pokora \cite[Remark 2.5]{Po} lists several other examples for which \eqref{eq:langer} is tight, but these do not appear to be part of an infinite family.

A crucial point is that Construction \ref{con:hesse} cannot be realized in $\R^2$, 
since it would violate the Sylvester--Gallai theorem.
The best construction in $\R^2$ that we are aware of is the following example, which is ascribed to B\"or\"oczky in \cite{CM}.

\begin{construction}[B\"or\"oczky's example]\label{con:boroczky}
For $n$ divisible by $2$, there is a set $P$ of $n$ points in $\R^2$ with $n/2$ points on a conic and $n/2$ points on a line that is disjoint from the conic,
satisfying the following properties.
For any two points of $P$ on the conic, the line through them hits the disjoint line in a point of $P$; this gives $\binom{n/2}{2}$ lines with three points.
The only other spanned lines are the $n/2$ tangent lines to the conic at points of $P$, each of which hits another point of $P$ on the disjoint line.
Thus we have
\[\sum_{i\geq 2}i\ell_i = \frac{n}{2}\cdot 2 + \binom{n/2}{2}\cdot 3 + 1\cdot \frac{n}{2} 
= \frac{3n(n+2)}{8}.\]
Each point of $P$ on the conic lies on $n/2$ spanned lines,
and the total number of spanned lines is $\binom{n/2}{2}+n/2+1\geq n^2/8$.
\end{construction}

See \cite[Section 2]{GT} for a more detailed description, including an analysis of the constructions one obtains by adding or removing points.

\newpage
Construction \ref{con:boroczky} shows that,
although it is possible that Theorem \ref{eq:langer} can be improved in $\R^2$,
one cannot expect more than the following.

\begin{conjecture}
Let $P$ be a set of $n$ points in $\R^2$,
with at most $n/2$ points collinear.
Then
\[\sum_{i\geq 2} i \ell_i\geq \frac{3}{8}n^2.\]
\end{conjecture}

If this conjecture holds, it would imply a further improvement to the weak Dirac conjecture.
At the same time, Construction \ref{con:boroczky} shows that such a bound cannot directly prove the strong Dirac conjecture. 

Note that the strong Dirac conjecture does hold for Construction \ref{con:boroczky},
because any point on the conic lies  on $n/2$ spanned lines.
Perhaps one could prove the strong Dirac conjecture by showing that Construction \ref{con:boroczky} is in some sense the only non-collinear example (up to projective equivalence and small modifications) in $\R^2$ with at most $n/2$ points collinear and with fewer than $n^2/2$ incidences.
Indeed, the next best examples that we know have roughly $n^2/2$ incidences, namely two lines with $n/2$ points each, 
and Sylvester's group construction on an irreducible cubic (see \cite[Proposition 2.6]{GT}).

\paragraph{Other applications.}
We think that Theorem \ref{thm:langer} will have many more applications in combinatorial geometry.
Let us mention two examples.

Fulek et al. \cite{F} proved that there exists a $c>0$ such that if a finite point set in $\R^2$ is not covered by two lines, then there are three points in the set such that all three lines spanned by them contain at most $c$ points (they call this a \emph{$c$-ordinary triangle}).
Their proof gave $c=12000$.
Dubroff \cite{D} finds a more efficient argument that, with the help of Theorem \ref{thm:langer} and some of the corollaries obtained below, proved the same statement with $c=11$.

De Zeeuw \cite{DZ} proves that for any $\alpha>0$ there exists $c_\alpha>0$ such that if $n$ points in $\R^3$ have at most $\alpha n$ points on a plane, then they span at least $c_{\alpha} n^2$ lines with exactly two points. 
This result can be proved without Theorem \ref{thm:langer},
but using it gives far better values for $c_\alpha$, and also simplifies the proof.

\newpage
\section{Spanned lines in \texorpdfstring{$\R^2$}{the real plane}}

We consider a theorem of Beck \cite[Theorem 3.1]{Be},
sometimes referred to as ``Beck's theorem of two extremes''.
It states that $n$ points either span a line with $c_1n$ points, or they span $c_2n^2$ points.
Beck gave $c_1 = 1/100$ and left $c_2$ unspecified, 
but his proof would clearly give a very small value;
Payne \cite[Theorem 2.4]{Pa} wrote out a proof with $c_1 = c_2 = 2^{-15}$.
Payne and Wood \cite{PW} gave a more refined argument that roughly gives $c_1=c_2= 1/100$ (it is not stated explicitly in the paper, but in \cite[Theorem 5]{PW} one can for instance put $\ell =1/100$).

Here we use Langer's inequality to obtain much better constants.
The idea of the proof can already be found in Kelly and Moser \cite[Equation 4.62]{KM}.

\begin{theorem}\label{thm:beckreal}
For any set $P$ of $n$ points in $\R^2$,
one of the following is true:
\begin{itemize}
\item There is a line that contains more than $\alpha n$ points of $P$, with $\alpha = (6+\sqrt{3})/9 > 0.85$;
\item There are at least $n^2/9$ lines spanned by $P$.
\end{itemize}
\end{theorem}
\begin{proof}
First assume that $P$ has at most $2n/3$ points collinear, so that by Theorem \ref{thm:langer} inequality \eqref{eq:langer} holds.
We can rearrange Melchior's inequality \eqref{eq:melchior} to
\[\sum_{i\geq 2} (3-i)\ell_i\geq 3,\]
and then add Langer's inequality \eqref{eq:langer} to get
\[\sum_{i\geq 2} i\ell_i + \sum_{i\geq 2} (3-i)\ell_i 
\geq \frac{n(n+3)}{3} + 3,\]
or equivalently
\begin{equation}\label{eq:manylines}
3|L(P)| =  3\sum_{i\geq 2} \ell_i \geq \frac{n^2+3n+9}{3}.
\end{equation}
Thus $|L(P)|\geq n^2/9$, which proves the second alternative.

Now suppose that $P$ has more than $2n/3$ points collinear.
Let $L$ be the line with more than $2n/3$ points of $P$;
say $|P\cap L| = \alpha n$ and $|P\backslash L| = (1-\alpha)n$.
Then we count one line for every choice of a point from $P\cap L$ and a point from $P\backslash L$, but we may overcount once for every pair of points from $P\backslash L$ (when the line through that pair hits $L$ in a point of $P$), so 
\[|L(P)| \geq \alpha n\cdot (1-\alpha )n - \binom{(1-\alpha )n}{2} \geq \left(-\frac{3}{2}\alpha^2 +2\alpha  - \frac{1}{2}\right)n^2.
\]
As long as we have $-3\alpha^2/2 +2\alpha - 1/2 \geq 1/9$, the second alternative of the theorem holds.
Solving for $\alpha$ shows that this is the case for $\alpha\leq (6+\sqrt{3})/9$.
Otherwise,
the first alternative holds.
\end{proof}



The best example in $\R^2$ that we know of, in terms of having few spanned lines while having less than $2n/3$ points collinear,
is Construction \ref{con:boroczky}, which spans roughly $n^2/8$ lines while having at most $n/2$ points collinear.
Thus it seems that a small improvement to Theorem \ref{thm:beckreal} is still possible (one would have to decrease $\alpha$ in the first alternative to make this possible).

\newpage
Beck \cite{Be} also proved the closely related statement that if $n$ points have exactly $n-k$ collinear, 
then they determine at least $c_3kn$ lines.
This solved a problem that was often mentioned by Erd\H os (see for instance \cite[Section 4]{Er}).
Again Beck did not specify the constant, 
but Payne and Wood \cite[Theorem 5]{PW} refined the proof to obtain $c_3 = 1/98$ (and Payne \cite{Pa} refined this to $1/93$).
Langer's inequality gives an improvement, using essentially the same proof as Beck.

\begin{corollary}
Let $P$ be a set of $n$ points in $\R^2$ with at most $n-k$ points contained in any line.
Then 
\[|L(P)| \geq \frac{1}{9}kn.\]
\end{corollary}
\begin{proof}
If the second alternative of Theorem \ref{thm:beckreal} holds, then we are done, since $n^2/9 \geq kn/9$.
Suppose that the first alternative holds, so we have $k\leq (1-\alpha)n$.
Then we count
\[ k(n-k) - \binom{k}{2} = k\left(n-k - \frac{k-1}{2}\right) \geq k\left(n- \frac{3}{2}k\right) \geq \frac{3\alpha-1}{2} kn > \frac{1}{9}kn\]
lines.
\end{proof}

Erd\H os specifically asked for $ckn$ lines with a constant $c$ independent of $k$ and $n$, 
and in that sense doing better than $c=1/9$ would require an improvement in Theorem \ref{thm:beckreal}.
In Sylvester's cubic curve construction (see \cite{GT}), we have $k= n-3$ (at most three points are collinear), and we have about $n^2/6$ spanned lines, which shows that we cannot do better than $c = 1/6$.


\bigskip

The following corollary gives more detailed information on the spanned lines with very few points.
The idea, due to Elliott \cite[proof of Theorem 2]{El}, 
is that in any non-collinear point set at least half the spanned lines have at most three points; 
see also Purdy and Smith \cite[Lemma 2.2]{PS} or Payne and Wood \cite[Observation 15]{PW}.

\begin{corollary}\label{cor:twothreereal}
Let $P$ be a set of $n$ points in $\R^2$ with at most $\alpha n$ points collinear, for $\alpha = (6+\sqrt{3})/9\approx 0.85$.
Then
\[ \ell_2+\ell_3 \geq \frac{n^2}{18}.\]
\end{corollary}
\begin{proof}
Writing Melchior's inequality \eqref{eq:melchior} as
\[\ell_2 \geq 3 + \sum_{i\geq 4} (i-3) \ell_i\]
and adding $\ell_2 + 2\ell_3$ on both sides  gives
\[2\ell_2+2\ell_3 \geq 3 + \ell_2 + 2\ell_3 + \sum_{i\geq 4} (i-3) \ell_i \geq 3 + \sum_{i\geq 2} \ell_i.\]
In other words, more than half the spanned lines have two or three points.
Thus, by Theorem \ref{thm:beckreal},
we have $\ell_2 + \ell_3 \geq (n^2/9)/2 = n^2/18$.
\end{proof}


\newpage
\section{Spanned lines in \texorpdfstring{$\C^2$}{the complex plane}}

As mentioned in the introduction, 
Melchior's inequality fails in $\C^2$, so the arguments of the previous section do not work there.
Nevertheless, 
we can obtain bounds that are slightly weaker than in $\R^2$, using only the Cauchy--Schwarz inequality (and Langer's inequality, of course).

Beck's theorem of two extremes is known to hold in $\C^2$, because it follows from the Szemer\'edi--Trotter theorem, which was proved in $\C^2$ by T\'oth \cite{T} and Zahl \cite{Z}.
However, their proofs would give extremely small constants in Beck's theorem.
Langer's inequality gives reasonable constants.

\begin{theorem}\label{thm:beckcomplex}
For any set $P$ of $n$ points in $\C^2$,
one of the following is true:
\begin{itemize}
\item There is a line that contains more than $\beta n$ points of $P$, with $\beta = (4+\sqrt{2})/6 >0.90$;
\item There are at least $n^2/12$ lines spanned by $P$.
\end{itemize}
\end{theorem}
\begin{proof}
We have
\[\binom{n}{2} 
= \sum_{\ell\in L(P)} \binom{|\ell\cap P|}{2}
 = \frac{1}{2}\sum_{\ell\in L(P)} |\ell\cap P|^2 -  \frac{1}{2}\sum_{\ell\in L(P)} |\ell\cap P|,
\]
so by the Cauchy-Schwarz inequality we get (writing $m =|L(P)|$ and $I = I(P,L(P))$)
\[ 2\binom{n}{2}
\geq \frac{1}{m} \left(\sum_{\ell\in L(P)} |\ell\cap P|\right)^2 -  \sum_{\ell\in L(P)} |\ell\cap P|
 = \frac{I^2}{m} - I
 =I\left( \frac{I}{m} - 1\right).
\]
Assuming that $P$ has at most $2n/3$ points collinear, 
Langer's inequality \eqref{eq:langer} gives
\[n^2-n \geq \frac{n(n+3)}{3}\left(\frac{n(n+3)}{3m} - 1\right),
\]
which is equivalent to
\begin{equation}\label{eq:spannedcomplex}
|L(P)| = m \geq \frac{(n+3)^2}{12}.
\end{equation}

On the other hand, if $P$ has more than $2n/3$ points collinear, then as in the proof of Theorem \ref{thm:beckreal}, if we have
$-3a^2/2 +2a - 1/2 \geq 1/12$, then the second alternative holds.
Solving for $a$ shows that this is the case for $a\leq (4+\sqrt{2})/6$.
Otherwise,
the first alternative holds.
\end{proof}

In Construction \ref{con:hesse} we have at most $n/3$ points collinear and $n^2/9+3$ spanned lines, which suggests that some improvement to Theorem \ref{thm:beckcomplex} is possible.

\bigskip

We would have liked to prove a complex analogue of Corollary \ref{cor:twothreereal} (see \cite[Section 4]{DZ} for an application where it would be very useful),
but we do not know how to do that without Melchior's inequality.
Here is an attempt that comes close.
By \eqref{eq:bojanowski1},
we have
\[\ell_2+\ell_3\geq \ell_2 + \frac{4}{3}\ell_3 \geq
\sum_{i\geq 5}\frac{i^2-4i}{4} \ell_i
\geq \sum_{i\geq 5}\ell_i,\]
so, adding $\ell_2+\ell_3+\ell_4$ on both sides, 
\[2(\ell_2+\ell_3) +\ell_4\geq  \sum_{i\geq 2}\ell_i = |L(P)|.\]
From \eqref{eq:basic} we have $6\ell_4\leq \binom{n}{2}\leq n^2/2$, so $\ell_4\leq n^2/12$,
which together with \eqref{eq:spannedcomplex} gives
\[\ell_2+\ell_3\geq \frac{1}{2}\left(|L(P)|-\ell_4\right) \geq \frac{1}{2}\left(\frac{n^2 + 6n +9}{12}-\frac{n^2}{12}\right)
=\frac{1}{4}n +\frac{3}{8}. \]

The problem seems to be that \eqref{eq:bojanowski1} gives no control over $\ell_4$.
This bound is actually worse than the bound $\ell_2+\ell_3\geq n$ that follows directly from inequality \ref{eq:hirzebruch} or \ref{eq:bojanowski1}, 
but we think this argument is worth  mentioning, 
because it shows that any bound of the form $\ell_4<cn^2$ with $c< 1/12$ would give a quadratic lower bound on $\ell_2+\ell_3$.
In $\R^2$, Brass \cite{Br} proved $\ell_4< n^2/14$ using Melchior's inequality, but again we have no equivalent in $\C^2$.

However,
we can obtain the following bound on the number of lines with five or more points.
A similar statement using \eqref{eq:hirzebruch} can be found in Payne \cite[Observation 3.18]{Pa}.

\begin{corollary}\label{cor:manypoorcomplex}
Let $P$ be a set of $n$ points in $\C^2$ with at most $2n/3$ points collinear.
Then for any $k\geq 5$ we have
\begin{equation}\label{eq:richlines}
\sum_{i \geq k} \ell_i \leq \frac{4}{(k-2)^2}|L(P)|.
\end{equation}
\end{corollary}
\begin{proof}
By \eqref{eq:bojanowski1} we have
\[\sum_{i=2}^{k-1} \ell_i \geq \ell_2 + \frac{3}{4}\ell_3 
\geq n + \sum_{i\geq 5} \frac{i^2-4i}{4}\ell_i
> \frac{k^2-4k}{4}\sum_{i\geq k} \ell_i,\]
so, adding $\frac{k^2-4k}{4}\sum_{i= 2}^{k-1} \ell_i$ on both sides, we get
\[\frac{k^2-4k+4}{4}\sum_{i= 2}^{k-1} \ell_i
> \frac{k^2-4k}{4}\sum_{i\geq 2} \ell_i.\]
Thus we have
\[\sum_{i = 2}^{k-1} \ell_i 
> \frac{(k-2)^2-4}{(k-2)^2}|L(P)|,\]
which implies the statement.
\end{proof}

It follows in particular, using Theorem \ref{thm:beckcomplex},
that
\[\ell_2+\ell_3+\ell_4 
> \frac{5}{9}|L(P)|\geq \frac{5}{108}n^2.\]

Since $|L(P)|\leq \binom{n}{2}$, 
\eqref{eq:richlines} implies that the number of lines with at least $k$ points satisfies
\[\sum_{i \geq k} \ell_i \leq \frac{2n^2}{(k-2)^2}.\]
But note that one can do better if one knows that the number of spanned lines is less than $\binom{n}{2}$ (see \cite{Du} for an application of this observation).

Compare this with the Szemer\'edi--Trotter theorem \cite{ST, T, Z} in $\R^2$, 
which implies that the number of lines with at least $k$ points is at most $cn^2/k^3$ for some constant $c$ (when $k<\sqrt{n}$).
This $c$ is fairly large 
(\cite{F} states $c=125$), so \eqref{eq:richlines} does better for small $k$ (specifically $k\leq 58$).


\end{document}